\documentclass{amsart}
\usepackage{amssymb}

\usepackage[left=2.8cm,right=2.8cm,top=2.9cm,bottom=2.9cm]{geometry}

\usepackage{amsfonts}
\usepackage{color}
\usepackage{graphicx}
\usepackage[bookmarksnumbered,colorlinks, linkcolor=blue, citecolor=red, pagebackref, bookmarks, breaklinks]{hyperref}

\newtheorem{thm}{Theorem}[section]
\newtheorem{cor}[thm]{Corollary}
\newtheorem{lema}[thm]{Lemma}
\newtheorem{prop}[thm]{Proposition}
\theoremstyle{definition}

\theoremstyle{remark}
\newtheorem{exam}[thm]{Example}

\newtheorem{rem}[thm]{Remark}
\numberwithin{equation}{section}
\newcommand{\R}{\mathbb R}
\newcommand{\N}{\mathbb N}
\newcommand{\Z}{\mathbb Z}

\def\C{\mathbf {C}}
\def\d{\mathbf {d}}
\def\J{{\mathcal{J}}}
\newcommand{\ve}{\varepsilon}
\newcommand{\lam}{\lambda}
\newcommand{\cd}{\rightharpoonup}
\def\pint{\operatorname {--\!\!\!\!\!\int\!\!\!\!\!--}}

\title[Eigenvalues and minimizers for a non-standard growth non-local operator]{Eigenvalues and minimizers for a non-standard growth non-local operator}
\author[A. M. Salort]
{Ariel M.  Salort}%

\address{Departamento de Matem\'atica, FCEyN - Universidad de Buenos Aires and
\hfill\break \indent IMAS - CONICET
\hfill\break \indent Ciudad Universitaria, Pabell\'on I (1428) Av. Cantilo s/n. \hfill\break \indent Buenos Aires, Argentina.}

\email[A.M. Salort]{asalort@dm.uba.ar}
\urladdr{http://mate.dm.uba.ar/~asalort}

\begin{document}



\begin{abstract}
In this article we study eigenvalues and minimizers of a fractional non-standard growth problem. We prove several properties on this quantities and their corresponding eigenfunctions.
\end{abstract}

\vspace*{-1cm}
\maketitle
\vspace{-1cm}
\setlength{\parskip}{0.06em}
\tableofcontents
\setlength{\parskip}{0.5em}

\vspace{-1cm}

\section{Introduction} \label{sec.autov}
In the last years the eigenvalue problem associated with the $p-$Laplacian operator 
\begin{align} \label{eq.p.lap}
\begin{cases}
-\Delta_p u:=-div(|\nabla u|^{p-2}\nabla u) =\lam |u|^{p-2}u &\quad \text{ in } \Omega,\\
u=0 &\quad \text{ on } \partial\Omega,
\end{cases}
\end{align}
has received a huge attention, where  $\Omega\subset \R^n$ is an open and bounded set and $p>1$.

Properties on the spectrum of \eqref{eq.p.lap} and its principal eigenvalue 
$$
\lam_1 :=\inf\{ \|\nabla u\|_{L^p(\Omega)} \colon \|u\|_{L^p(\Omega)}=1\}
$$
have been widely studied and generalized, and a vast bibliography is available. We refer for instance the pioneering  works of Anane \cite{An}, Allegreto and Huang \cite{AlHu}, Lindqvist \cite{Lindq},  Anane and Tsouli \cite{AnTs} and references to them. The generalization to homogeneous monotone operators of the form $-div(a(x,\nabla u))$ has been dealt for instance by Kawohl et al. \cite{KLP} and Fern\'andez Bonder et al. \cite{FBPS}.
The extension to operators involving  behaviors  more general than powers was treated by several authors: the eigenvalue problem related with the \emph{$g-$Laplacian} defined as $\Delta_g u =  div(g(|\nabla u|) \nabla u)$, where $g$ is a positive nondecreasing function, was studied by 
Gossez  and  Mansevich in  \cite{GM}, Garc\'ia-Huidobro et al. in \cite{GLMS} and Mustonen and Tienari in \cite{MT}, for instance. In the same spirit, in \cite{Montenegro} Montenegro studies a related minimization problem.

Eigenvalue problems have been also treated in   nonlocal settings. In \cite{SV} Servadai and Valdonoci as well as Kwa\'snicki in \cite{K} study the spectrum of different non-local linear operators. In \cite{LL} Lindqvist and Lindgren   define and study properties of the first eigenvalue of the \emph{fractional $p-$Laplacian}
$$
(-\Delta_p)^s u(x):=2 \text{p.v.} \int_{\R^n} \frac{|u(x)-u(y)|^{p-2}(u(x)-u(y))}{|x-y|^{n+sp}} \,dy
$$
where $s\in(0,1)$ and $p>1$.
 
The main aim of this manuscript is to study eigenvalues and minimizers involving the non-local non-linear non-homogeneous operator 
$$
(-\Delta_g)^s u:=2 \text{p.v.} \int_{\R^n} g\left( |D_s u|\right)\frac{D_s u}{|D_s u|} \frac{dy}{|x-y|^{n+s}},
$$
defined in \cite{FBS}, where the $s-$H\"older quotient is defined as
$$
D_s u(x,y) = \frac{	u(x)-u(y)}{|x-y|^s}.
$$
Here p.v. stands for {\em in principal value}, $s\in (0,1)$ is a fractional parameter and $g$ is a positive non-decreasing function such that $g=G'$, being $G$ an   function belonging to the so-called Young class (see Section \ref{prelim} for details) satisfying the growth condition
$$
1<p^-\leq \frac{tg(t)}{G(t)} \leq p^+<\infty \quad \forall t>0
$$
for some constants $p^\pm$.

Given an open and bounded set $\Omega\subset \R^n$ and $\lam\in\R$ we consider the problem
\begin{align} \label{eq.autov}
\begin{cases}
(-\Delta_g)^s u= \lam   g(|u|)\frac{u}{|u|} &\quad \text{ in } \Omega,\\
u=0 &\quad \text{ on } \R^n \setminus \Omega.
\end{cases}
\end{align}

In this context we  say that $\lam$ is an \emph{eigenvalue}  of  \eqref{eq.autov} with    \emph{eigenfunction} $u$ belonging to the fractional Orlicz-Sobolev space $W^{s,G}_0(\Omega)\setminus\{ 0\}$ (see Section \ref{prelim} for details)  provided that 
\begin{equation}\label{eq.autov.debil} 
\langle (-\Delta_g)^s u,v \rangle:=\frac12 \iint_{\R^n\times\R^n} g(|D_s u|) \frac{D_s u}{|D_s u|} D_s v \,d\mu = \lam \int_\Omega  g(|u|)\frac{u}{|u|}v\,dx
\end{equation}
holds for all $v\in W^{s,G}_0(\Omega)$, where we have denoted  the measure
$d\mu(x,y)=\frac{dxdy	}{|x-y|^n}$.
 
The \emph{spectrum} $\Sigma$ is defined as the set
$$
\Sigma :=\{\lam\in\R \colon \text{there exists }u\in W^{s,G}_0(\Omega) \text{ nontrivial solution to }\eqref{eq.autov.debil} \}.
$$
Problem \eqref{eq.autov} is the the Euler-Lagrange equation corresponding to the minimization problem 
\begin{equation} \label{min.prob}
\alpha_{1,\mu}=\inf_{u\in M_\mu}  \frac{\mathcal{F}(u)}{\mathcal{G}(u)} \quad \text{ with } \quad M_\mu = \{u\in W^{s,G}_0(\Omega):\mathcal{G}(u) =\mu\},
\end{equation}
where functionals $\mathcal{F},\mathcal{G}:W^{s,G}_0(\Omega) \to \R$ are defined by
\begin{equation} \label{funcionales}
\mathcal{F}(u)=\iint_{\R^n\times\R^n} G(|D_s u|)\,d\mu, \qquad \mathcal{G}(u)=\int_\Omega   G(|u|)\,dx.
\end{equation}

By means of the direct method of the calculus of variations, in Proposition \ref{propo.1} it is proved that for each election of $\mu>0$, the minimization problem \eqref{min.prob} is attained for a function $u_{1,\mu}\in W^{s,G}_0(\Omega)$. Moreover, since $\mathcal{F}$ and $\mathcal{G}$ are Fr\'echet differentiable due to Proposition \ref{propo.el}, by the Lagrange multipliers method, Theorem \ref{teo666} states that there exists a number $\lam_{1,\mu}\in \R$ being an eigenvalue of \eqref{eq.autov} with associated eigenfunction $u_{1,\mu}$, i.e.
\begin{equation} \label{lam.1}
\langle (-\Delta_g)^s u_{1,\mu},v \rangle = \lam_{1,\mu} \int_\Omega   g(|u_{1,\mu}|)\frac{u_{1,\mu}}{|u_{1,\mu}|}v\,dx \quad \forall v\in W^{s,G}_0(\Omega).
\end{equation}
In contrast with $p-$Laplacian type problems, $\alpha_{1,\mu}$ may differ from $\lam_{1,\mu}$, although both quantities are comparable: in Corollary \ref{prop.cota.inf} it is proved that there are constants $c_1, c_2>0$ independent on $\mu$ such that
\begin{equation} \label{desig.intr}
0<c_1 \alpha_{1,\mu} \leq \lam_{1,\mu} \leq c_2 \alpha_{1,\mu}.
\end{equation}
We remark that the number $\alpha_{1,\mu}$ can be seen as the best Poincar\'e's constant in $W^{s,G}_0(\Omega)$, and, in general, it is not an eigenvalue. Moreover,  the eigenvalue $\lam_{1,\mu}$ in general does not admit a variational characterization.

Due to the possible lack of homogeneity of \eqref{eq.autov}, the numbers  $\alpha_{1,\mu}$ and $\lam_{1,\mu}$ strongly depend on the energy level $\mu$. Therefore,  we can consider the less quantities  over all possible choices of  $\mu$. We define
\begin{equation} \label{alpha.lam.1}
\lam_1=\inf\{\lam_{1,\mu} : \mu>0\}, \qquad \alpha_1=\inf\{\lam_{1,\mu} : \mu>0\}.
\end{equation}
Since in Proposition \ref{espectro.cerrado} we prove that $\Sigma$ is a closed set, it is derived in Corollary \ref{es.autov} that $\lam_1$ is in fact an eigenvalue of \eqref{eq.autov}. Furthermore, an inequality of the type \eqref{desig.intr} still being true between $\alpha_1$ and $\lam_1$.

With regard to higher eigenvalues $\lam$   with continuous sign changing eigenfunction $u$,   in Proposition \ref{prop.cota.inf.1} it is stated the following relation
$$
p^+\lam(\Omega)>\lam_1(\Omega^+), \quad p^+\lam(\Omega)>\lam_1(\Omega^-),
$$
where $\Omega^+$ and $\Omega^-$ denote the subset of $\Omega$ where $u>0$ and $u<0$, respectively.

An important property the eigenfunctions $u_{1,\mu}$  of $\lam_{1,\mu}$ is established in Theorem \ref{coro.1}: it is one-signed in $\Omega$ whenever it is a continuous function.
 
In an analogous way, one could multiply the right side of \eqref{eq.autov} by a weight function $\rho$. Hence, given a function $\rho$ satisfying
\begin{equation} \label{cond.rho}
0<\rho_- \leq \rho(x) \leq \rho_+ < \infty \qquad \forall x\in \R^n
\end{equation}
for certain constant $\rho_\pm$, and $\mu>0$, one can  consider the corresponding quantity $\alpha_{1,\mu}(\rho)$ defined as
\begin{equation} \label{alfas}
\alpha_{1,\mu}(\rho)=\inf_{u\in M_\mu}  \frac{\mathcal{F}(u)}{\int_\Omega \rho G(|u|)\,dx}.
\end{equation}
In Theorem \ref{sin.orden} we prove that $\alpha_{1,\mu}(\rho)$ is continuous with respect to   $\rho$. Namely, if $\mu>0$ is fixed and  $\{\rho_\ve\}_{\ve>0}$ is a sequence of functions satisfying \eqref{cond.rho} such that  $\rho_\ve \cd \rho_0$ weakly* in $L^\infty(\Omega)$, then it holds that
$$
\lim_{\ve\to 0} \alpha_{1,\mu}(\rho_\ve) = \alpha_{1,\mu}(\rho_0) .
$$
Moreover, when the family $\{\rho_\ve\}_{\ve>0}$ is $Q-$periodic, being $Q$ the unit cube in $\R^n$, then the rate of the convergence can be estimated.  Indeed, in this case $\rho_0=\pint_Q \rho$ and in Theorem \ref{con.orden} we prove that
$$
|\alpha_{1,\mu}(\rho_\ve) -\alpha_{1,\mu}(\rho_0) | \leq C \ve^{sp^+} (\alpha_{1,\mu})^2
$$
where $C$ is a constant independent of $\ve$ and $\mu$.

Finally, in Proposition \ref{sto1},   through a $\Gamma-$convergence argument we prove  that
$$
\lim_{s\uparrow 1} \alpha_{1,\mu,s} = \alpha_{1,\mu,1}:=\inf\left\{\frac{\Phi_{\tilde G}(|\nabla u|)}{\Phi_{\tilde G}(u)} \colon u\in W^{1,\tilde G}_0 (\Omega),  \Phi_{\tilde G}(u)=\mu     \right\} 
$$
where we have stressed the dependence on $s$ in $\alpha_{1,\mu,s}$, and here $\tilde G$ is a suitable limit Young function explicitly given in terms of $G$. Observe that $\alpha_{1,\mu,1}$ is a minimizer of the well-known local operator $\tilde g-$Laplacian, being $\tilde g=\tilde G'$.

The paper of organized as follows: in Section \ref{prelim} we introduce the class of Young functions and some useful properties on them as well as the fractional Orlicz-Sobolev spaces. In Section \ref{poinc.sec} we prove some Poincar\'e's type inequalities and maximum principles. Section \ref{sec.eig} is devoted to study the eigenvalue problem \eqref{eq.autov} whilst Section \ref{min.sec} is dedicated to treat  the corresponding minimizers. Finally, in Section \ref{sec.adic} some further  results are provided.

\section{Preliminary results} \label{prelim}
In this section we introduce the classes of Young function and fractional Orlicz-Sobolev functions as well as the fractional $g-$Laplacian.
\subsection{Young functions}
We say that a function $G:\R_+\to \R_+$ belongs to the \emph{Young class} if it admits the integral formulation $G(t)=\int_0^t g(s)\,ds$, where the right continuous function $g$ defined on $[0,\infty)$ has the following properties:
\begin{align*} 
&g(0)=0, \quad g(t)>0 \text{ for } t>0 \label{g0} \tag{$g_1$}, \\
&g \text{ is nondecreasing on } (0,\infty) \label{g2} \tag{$g_2$}, \\
&\lim_{t\to\infty}g(t)=\infty  \label{g3} \tag{$g_3$} .
\end{align*}
From these properties it is easy to see that a Young function $G$ is continuous, nonnegative, strictly increasing and convex on $[0,\infty)$. Without loss of generality $G$ can be normalized such that $G(1)=1$.

The \emph{complementary Young function} $G^*$ of a Young function $G$ is defined as
$$
G^*(t)=\sup\{tw -G(w): w>0\}.
$$ 
From this definition the following Young-type inequality holds
\begin{equation} \label{Young}
st\leq G(s)+G^*(t)\qquad \text{for all }s,t\geq 0.
\end{equation}
Moreover, it is not hard to see that $G^*$ can be written in terms of the inverse of $g$ as
\begin{equation} \label{xxxx}
G^*(t)=\int_0^t g^{-1}(s)\,ds,
\end{equation}
see \cite[Theorem 2.6.8]{RR},

The following growth condition on the Young function $G$ will be assumed 
\begin{equation} \label{cond} \tag{L}
1<p^-\leq \frac{tg(t)}{G(t)} \leq p^+<\infty \quad \forall t>0
\end{equation}
where $p^\pm$ are fixed numbers.

The following  properties are well-known in the theory of Young function. We refer, for instance, to the books \cite{KJF} and \cite{RR} for an introduction to Young functions and Orlicz spaces, and the proof of these results. See also \cite{FBPLS}.
\begin{lema} 
Let $G$ be a Young function satisfying \eqref{cond} and $s,t\geq 0$. Then 
\begin{align*}
  &\min\{ s^{p^-}, s^{p^+}\} G(t) \leq G(st)\leq   \max\{s^{p^-},s^{p^+}\} G(t),\tag{$G_1$}\label{G1}\\
  &G(s+t)\leq \C (G(s)+G(t)) \quad \text{with } \C:=  2^{p^+},\tag{$G_2$}\label{G2}\\
	&G \text{ is Lipschitz continuous: } |G(s)-G(t)| \leq |g(s)| |s-t|. \tag{$G_3$}\label{G3}
 \end{align*}
\end{lema}
Condition \eqref{G2} is known as the \emph{$\Delta_2$ condition} or \emph{doubling condition} and, as it is showed in \cite[Theorem 3.4.4]{KJF}, it is equivalent to the right side inequality in \eqref{cond}.

It is easy to see that condition \eqref{cond} implies that
\begin{equation}
(p^+)'\leq \frac{t(G^*)'(t)}{G(t)} \leq (p^-)' \quad \forall t>0,\tag{$G^*_3$}
\end{equation}
from where it follows that $G^*$ also  satisfies the $\Delta_2$ condition.
\begin{lema} 
Let $G$ be a Young function satisfying \eqref{cond} and $s,t\geq 0$. Then 
\begin{equation} \label{g.s.1} \tag{$G^*_1$}
  \min\{s^{(p^-)' },s^{(p^+)' }\} G^*(t)\leq G^*(st)\leq    \max\{s^{(p^-)'},s^{(p^+)'}\} G^*(t),
\end{equation}
where $(p^\pm)' = \frac{p^\pm}{p^\pm -1}$.	
\end{lema}

Since $g^{-1}$ is increasing, from \eqref{xxxx} and \eqref{cond} it is immediate the following relation.
\begin{lema} \label{lemita}
Let $G$ be an Young function satisfying \eqref{cond} such that $g=G'$ and denote by  $G^*$ its complementary function. Then
$$
G^*(g(t)) \leq p^+ G(t)
$$
holds for any $t\geq 0$.
\end{lema}

\begin{exam}\label{ej.orlicz}
The family of Young functions includes the following examples.
\begin{enumerate}
\item \emph{Powers}.
If $g(t)=t^{p-1}$, $p>1$ then $G(t)=\frac{t^p}{p}$, and $p^\pm = p-1$.

\item \emph{Powers$\times$logarithms}. Given $b,c>0$ if $g(t)=t\log(b+ct)$ then
$$
G(t)=\frac{1}{4c^2}\left(ct(2b-ct)-2(b^2 -c^2 t^2) \log(b+ct) \right)
$$
and $p^-=2$, $p^+=3$. In general, if $a,b,c>0$ and $g(t)=t^a\log(b+ct)$ then
$$
G(t)=\frac{t^{1+a}}{(1+a)^2}\left( {}_2 F_1(1+a,1,2+a,-\tfrac{ct}{b})  + (1+a)\log(b+ct)-1 \right)
$$
with $p^-=1+a$, $p^+=2+a$, where ${}_2 F_1$ is a hyper-geometric function.

\item \emph{Different powers behavior}. An important example is the family of functions $G$ allowing different power behavior near $0$ and infinity. The function $G$ can be considered  such that 
$$
g\in C^1([0,\infty)), \quad g(t)=c_1 t^{a_1} \text{ for } t\leq s \quad \text{ and } \quad  g(t)=c_2 t^{a_2}+d \text{ for } t\geq s.
$$
In this case $p^-=1+\min\{a_1,a_2\}$ and $p^+=1+\max\{a_1,a_2\}$.

\item \emph{Linear combinations}. If $g_1$ and $g_2$ satisfy \eqref{cond} then $a_1 g_1 + a_2 g_2$ also satisfies \eqref{cond} when $a_1,a_2\geq 0$.

\item  \emph{Products}. If $g_1$ and $g_2$ satisfy \eqref{cond} with constants $p^\pm_i$, $i=1,2$, then $g_1g_2$ also satisfies \eqref{cond} with constants $p^-=p^-_1+p^-_2 -1$ and $p^+=p^+_1+p^+_2 -1$.

\item \emph{Compositions}. If $g_1$ and $g_2$ satisfy \eqref{cond} with constants $p_i^\pm$, $i=1,2$, then $g_1\circ g_2$ also satisfies \eqref{cond} with constants $p^-=1+(p^-_1 -1)(p^-_2-1)$ and $p^+=1+(p^+_1-1)(p^+_2-1)$.
\end{enumerate}
\end{exam}   

\subsection{Fractional Orlicz-Sobolev spaces}

Given a Young function $G$, a fractional parameter $s\in(0,1)$ and an open and bounded set $\Omega\subseteq \R^n$,  we consider the following spaces:
\begin{align*}
&L^G(\Omega) :=\left\{ u\colon \R^n \to \R \text{ Lebesgue  measurable, such that }  \Phi_{G}(u) < \infty \right\},\\
&W^{s,G}(\Omega):=\left\{ u\in L^G(\Omega) \text{ such that } \Phi_{s,G}(u)<\infty \right\},
\end{align*}
where the modulars $\Phi_G$ and $\Phi_{s,G}$ are defined as
$$
\Phi_{G}(u):=\int_{\Omega} G(|u(x)|)\,dx,\qquad 
\Phi_{s,G}(u):=
  \iint_{\R^n\times\R^n} G( |D_su(x,y)|)  \,d\mu,
$$
with the \emph{$s-$H\"older quotient} defined as
$$
D_s u(x,y)=\frac{u(x)-u(y)}{|x-y|^s},
$$
and $d\mu(x,y):=\frac{ dx\,dy}{|x-y|^n}$.
These spaces are endowed with the so-called \emph{Luxemburg norms}
$$
\|u\|_G := \inf\left\{\lambda>0\colon \Phi_G\left(\frac{u}{\lambda}\right)\le 1\right\}, \qquad 
\|u\|_{s,G} := \|u\|_G + [u]_{s,G},
$$
where the  {\em $(s,G)$-Gagliardo semi-norm} is defined as 
$$
[u]_{s,G} :=\inf\left\{\lambda>0\colon \Phi_{s,G}\left(\frac{u}{\lambda}\right)\le 1\right\}.
$$
We also consider the following space
$$
W^{s,G}_0(\Omega) := \{u\in W^{s,G}(\R^n) \colon u=0 \text{ a.e. in } \R^n\setminus \Omega \}.
$$
Observe that the following inclusions hold
$$
W^{s,G}_0(\Omega)\subset W^{s,G}(\R^n)\subset L^G(\R^n).
$$

Hereafter, $\Omega$ will always stand  for a bounded open set in $\R^n$ whose diameter is  denoted as
$$
\d=\text{diam}(\Omega)=\sup\{|x-y|:x,y\in\Omega\}.
$$
 
We finish this section recalling some useful results on fractional Orlicz-Sobolev spaces.
 
\begin{prop}[\cite{FBS}, Proposition 2.10] \label{prop.WsG}
Let $s\in(0,1)$ and $G$ a Young function satisfying \eqref{cond}.  Then $W^{s,G}(\R^n)$ is a reflexive and separable Banach space.  Moreover, $C^\infty_c(\R^n)$ is dense in $W^{s,G}(\R^n)$.
\end{prop}
A variant of the well-known Fr\'echet-Kolmogorov compactness theorem gives the compactness of the inclusion of $W^{s,G}$ into $L^G$.
\begin{prop}  [\cite{FBS}, Theorem 3.1] \label{teo.comp}
Let $s\in(0,1)$ and $G$ a Young function satisfying \eqref{cond}. Then $W^{s,G}(\Omega)\Subset L^G(\Omega)$.
\end{prop}
Another useful result regarding strong convergence is the following.
\begin{prop} [\cite{RR}, Theorem 12] \label{teo.rr}
Let $\{u_n\}_{n\in\N}$ be a sequence in $L^G$ and $u\in L^G$. If $G^*$ satisfies the $\Delta_2$ condition, $\Phi_G(u_n)\to \Phi_G(u)$  and  $u_n\to u$ a.e., then $u_n\to u$ in the $L^G$ norm.
\end{prop}
Finally we recall that fractional Orlicz-Sobolev spaces are embedded into the usual fractional Sobolev spaces as well.

\begin{prop} \cite[Corollary 2.10]{FBPLS} \label{inclusion}
Given $0<t<s<1$ and a Young function $G$ satisfying \eqref{cond}, for any $q$ such that  $1\leq q<p^-$ it holds that $
W^{s,G}_0(\Omega)\subset W^{t,q}_0(\Omega)
$ with continuous inclusion.
\end{prop}
As a consequence, since   $W^{s,p^-}_0(\Omega)$ is continuously embedded into $C^{0,\alpha}(\Omega)$ for  $\alpha=s-\frac{n}{p^-}>0$, see \cite[Section 8]{hit}, we can characterize  continuous functions in fractional Orlicz-Sobolev spaces.
 
\begin{cor} \label{continua}
Let $\Omega\subset \R^n$ be a bounded and open set and $s\in(0,1)$. If $G$ is a Young function satisfying \eqref{cond} such that $s p^->n$, then $W^{s,G}_0(\Omega)\subset C^{0,\alpha}(\Omega)$ with  $\alpha=s-\frac{n}{p^-}$. 
\end{cor}
 
\subsection{The fractional $g-$Laplacian operator}

Let $G$ be a Young function and $s\in(0,1)$ be a  parameter. The fractional $g-$Laplacian operator is defined as 
\begin{align} \label{vp}
\begin{split}
(-\Delta_g)^s u&:=2 \text{p.v.} \int_{\R^n} g( |D_s u|) \frac{D_s u}{|D_s u|} \frac{dy}{|x-y|^{n+s}},
\end{split}
\end{align}
where p.v. stands for {\em in principal value} and $g=G'$. This operator can be seen as the gradient of the modular $\Phi_{s,G}(u)$ and  is well defined between $W^{s,G}(\R^n)$ and its dual space $W^{-s,G^*}(\R^n)$. In fact, in \cite[Theorem 6.12]{FBS} the following representation formula  is provided
$$
\langle (-\Delta_g)^s u,v \rangle = \frac12 \iint_{\R^n\times\R^n} g(|D_s u|) \frac{D_s u}{|D_s u|}  D_s v \,d\mu,
$$
for any $v\in W^{s,G}(\R^n)$.

\section{Some useful results on fractional Orlicz-Sobolev spaces} \label{poinc.sec}

In this section we probe two Poincar\'e's inequalities and a maximum principle in the context of nonlocal Orlicz-Sobolev spaces.

\subsection{Poincar\'e's inequalities}
We start this section proving a modular  inequality for small cubes. We will denote $(u)_Q$ the average of $u$ on $Q$.

\begin{lema} \label{poincarelema}
	Let $Q$ be the unit cube in $\R^n$, $n\geq 1$  and let $G$ be a Young function.  Then, for every $u\in W^{s,G}(Q_\ve)$ we have that
	$$
		\int_{Q_\ve} G(|u - (u)_{Q_\ve}|) \,dx \le c \ve^{sp^+}  \iint_{Q_\ve\times Q_\ve} G(|D_s u|) \,d\mu
	$$
	where $0<\ve\leq 1$, $Q_\ve = \ve Q$ and $c$ is a constant depending only on $n$.
\end{lema}
\begin{proof}
	Given $u\in W^{s,G}(Q_\ve)$, by using  Jensen's inequality it follows that
	\begin{align*}
		\int_{Q_\ve} G(|u-(u)_{Q_\ve}|) dx &= \int_{Q_\ve} G\left( \left| \pint_{Q_\ve} (u(x)-u(y)) \, dy \ \right|\right)  dx\\
		&\leq 
		 \int_{Q_\ve}  \pint_{Q_\ve} G(|u(x)-u(y)|) \, dy \,   dx\\
		 &\leq c\ve^{sp^+} \int_{Q_\ve}  \int_{Q_\ve} G\left(\frac{|u(x)-u(y)|}{|x-y|^{s}} \right) \frac{dxdy}{|x-y|^n},
	\end{align*}
	from where the result follows.
\end{proof}

The following Poincar\'e's inequality for modulars in $W^{s,G}_0(\Omega)$ gives as a consequence that $[\,\cdot\,]_{s,G}$ is an equivalent norm  in $W^{s,G}_0(\Omega)$.

\begin{prop} \label{thm.poincare}
Let $\Omega\subset \R^n$ be open and bounded and let $G$ be a Young function satisfying \eqref{cond}. Then for $s\in(0,1)$ it holds that
$$
\Phi_G(u) \le   \Phi_{s,G}(C_p \d^s u)$$
for all $u\in W^{s,G}_0(\Omega)$, where $C_p=\left( \frac{s p^+}{n\omega_n}\right)^\frac{1}{p^-}$ with  $\omega_n$  standing for the volume of the unit ball in $\R^n$.
\end{prop}

\begin{proof}
Let $C_p$ be a positive constant to determinate. Given $x\in \Omega$, observe that when $|x-y|\geq \d$, then $y\notin \Omega$. Hence, by using \eqref{G1} we get
\begin{align*}
\Phi_{s,G}(C_p\d^s u) &\geq 
\int_\Omega \int_{|x-y|\geq \d} G\Big( \frac{  \d^s  }{|x-y|^s} C_p |u(x)|  \Big) \frac{dydx}{|x-y|^n} \\
&\geq 
C_p ^{p^-} \d^{sp^+} \left(\int_\Omega  G(|u(x)|)  dx\right) \left( \int_{|z|\geq \d} \frac{dz}{|x-y|^{n+sp^+}} \right)
\end{align*}
since, without loss of generality we can assume that $C_p\geq 1$.

Now, by using polar coordinates we have that
$$
 \int_{|z|\geq \d} \frac{dz}{|x-y|^{n+sp^+}} = \frac{n\omega_n}{sp^+} \d^{-sp^+},
$$
and the result follows choosing   properly the constant $C_p$.
\end{proof}
As a direct implication we obtain an inequality for norms.

\begin{cor} \label{poincare.norma}
Under the same assumptions than in Theorem \ref{thm.poincare}, it holds that
$$
\|u\|_{G} \leq  C_p \d^s  [u]_{s,G}
$$
for every $s\in(0,1)$ and $u\in W^{s,G}_0(\Omega)$.

\end{cor}
\begin{proof}
Given $u\in W^{s,G}_0(\Omega)$, applying Theorem \ref{thm.poincare} to the function $u/C_p \d^s[u]_{s,G}$, we get
$$
\Phi_G\left(\frac{u}{C_p \d^s [u]_{s,G}}\right) \leq \Phi_{s,G}\left( \frac{u}{[u]_{s,G}}\right) = 1
$$
by definition of the Luxemburg's norm. Consequently, 
$$
\|u\|_G = \inf\{\lam: \Phi_G\left(\tfrac{u}{\lam}\right)\leq 1\} \leq  C_p \d^s [u]_{s,G}
$$
as desired.
\end{proof}
Condition \eqref{G1} on Proposition \ref{thm.poincare} gives the following inequality.
\begin{cor} \label{coro.cota.inf}
Under the same assumptions than in Theorem \ref{thm.poincare}, it holds that
$$
\Phi_G(u) \le C \max\{\d^{sp^+},\d^{sp^-}\} \Phi_{s,G}(  u)
$$
for all $s\in(0,1)$ and  $u\in W^{s,G}_0(\Omega)$, where $C=C(s,n,p^\pm)$.
\end{cor}

\subsection{A strong maximum principle for continuous solutions}
In order to define the  main result in this paragraph it is convenient to define the notion of weak and viscosity solutions in  our settings. Given an open and bounded set $\Omega\subset\R^n$ and $f\in L^{G^*}(\Omega)$,  consider the following Dirichlet equation
\begin{align} \label{eq.f}
\begin{cases}
(-\Delta_g)^s u= f &\quad \text{ in }\Omega,\\
u=0 &\text{ in }\R^n\setminus \Omega.
\end{cases}
\end{align}
We say that $u\in W^{s,G}_0(\Omega)$ is a  \emph{weak sub-solution (super-solution)} to \eqref{eq.f} if
\begin{equation*}  
\langle (-\Delta_g)^s u, v\rangle \le (\ge) \int_\Omega fv   \qquad \text{for all non-negative }v\in W^{s,G}_0(\Omega).
\end{equation*}
If $u$ is simultaneously a weak super- and sub-solution, then we say that $u$ is a \emph{weak solution} to \eqref{eq.f}.

We say that an upper (lower) semi-continuous function $u$ such that $u\le 0$ ($u\geq 0$) in $\R^n\setminus\Omega$ is a  \emph{viscosity sub-solution (super-solution)} to \eqref{eq.f} if
whenever $x_0\in\Omega $ and $\varphi\in C^1_c(\R^n)$
are such that
$$
(i)\quad \varphi(x_0)=u(x_0), \qquad (ii)\quad u(x)\le (\geq) \varphi(x) \text{ for }x\ne x_0
$$
 then $(-\Delta_g)^s \varphi(x_0)\leq (\geq) f(x_0)$.

Finally, a continuous function $u$  is a \emph{viscosity solution} to $(\ref{eq.f})$ if it is a viscosity super-solution and a viscosity sub-solution.

\begin{rem}\label{rem.cte} Since $(-\Delta_g)^s  (\varphi+C)=(-\Delta_g)^s \varphi$, the previous definitions are equivalent if  the function $\varphi(x)+C$ (or $\psi(x)-C$) touches $u$ from below  (from above, respectively) at $x_0$.

 Furthermore, in the previous definitions we may assume that the test function touches $u$ strictly. Indeed, for a test function $\varphi$ touching $u$ from below, consider the function $h(x)=\varphi(x)- \eta(x)$, where $\eta\in C^\infty_c(\R^n)$ satisfies $\eta(x_0)=0$ and $\eta(x)>0$ for $x\ne x_0$. Notice that $h$ touches $u$ strictly.
Moreover, since the function $g$ is increasing it holds that $(-\Delta_g)^s h(x_0)\le (-\Delta_g)^s \varphi(x_0)$. For further details about general theory of viscosity solutions we refer, for instance, to the classical monographs \cite{CIL, IL}.
\end{rem}

The theory of viscosity solutions is based on a point-wise testing; by \cite[Lemma  2.17]{FBPLS}, $(-\Delta_g)^s$  is well defined point-wisely for any test function $\varphi\in C^1(\R^n)\cap L^\infty(\R^n)$ and for every $x\in\R^n$ provided that 
$$
(1-s)p^->1.
$$
Moreover,   in light of Corollary \ref{continua}, weak solutions are continuous when
$$
sp^->n.
$$
Therefore, in order to deal with viscosity solutions coming from continuous weak solutions, we will impose  a lower bound for the growth of the Young function $G$ satisfying \eqref{cond}, namely, 
\begin{equation} \label{cond1} \tag{S}
p^-> \frac{n}{s(1-s)}.
\end{equation} 
Under this assumptions, weak and viscosity solutions can be related.
\begin{prop}\cite[Lemma 3.7]{FBPLS} \label{equival}
Let $\Omega\in \R^n$ be open and bounded and let $G$ be a Young function satisfying  \eqref{cond} and \eqref{cond1}. Then,  a weak solution $u\in W^{s,G}_0(\Omega)$ of \eqref{eq.f} is a viscosity solution of \eqref{eq.f}.
\end{prop}

We state  the following weak maximum principle.

\begin{prop} \label{weak.max}
Let $\Omega\subset \R^n$ be open and bounded and let $G$ be a Young function. Then, if $f\geq 0$ in $\Omega$, then a weak solution $u\in W^{s,G}_0(\Omega)$ of $(-\Delta_g)^s u=0$  satisfies $u\geq 0$ in $\Omega$.
\end{prop}
\begin{proof}
Let $u^+:=\max\{u,0\}$ and $u^-=\max\{-u,0\}$ be the positive and negative parts of $u$, respectively. Testing with $u^-\in W^{s,G}_0(\Omega)$ we have
\begin{align*}
0&\leq \int_\Omega f u^- = \iint_{\R^n\times\R^n} g(|D_s u|) \frac{D_s u}{|D_s u|} D_s u^- \,d\mu\\
&=  \iint_{\R^n\times\R^n} D_s u D_s u^- \left(\int_0^1 g'((1-t)|D_s u|) \,dt\right) \,d\mu.
\end{align*}
Observe that
\begin{align*}
D_s u(x,y) D_s u^- (x,y)&=-\frac{u(x)u^-(y)+ u(y)w^-(x)}{|x-y|^s} -(D_s u^- (x,y))^2\\
&\leq -(D_s u^- (x,y))^2 \leq 0.
\end{align*}
Since $g'\geq 0$ and $g'(t)=0$ if and only if $t=0$, from the last two relations we get that $u^-\equiv 0$ in $\Omega$.
\end{proof}

Moreover, the following strong maximum principle for continuous functions holds.
\begin{prop} \label{u.positiva}
Let $\Omega\in \R^n$ be open and bounded and assume that $G$ fulfills \eqref{cond} and  \eqref{cond1}. If $u\in W^{s,G}_0(\Omega)$ is weak super-solution of $(-\Delta_g)^su=0$ in $\Omega$, then either $u>0$ in $\Omega$ or $u\equiv 0$.
\end{prop}
\begin{proof}
Let $u\in W^{s,G}_0(\Omega)$ be a weak super-solution of $(-\Delta_g)^su=0$ in $\Omega$. We can assume that $u\geq 0$ in $\Omega$. In light of condition \eqref{cond1}, from Proposition \ref{equival} we have that $u$ is also a viscosity super-solution of the same equation.

Let $x_0\in\Omega$ be a point where  $u(x_0)=0$. By definition of viscosity super-solution, for any test function $\varphi\in C^1_c(\R^n)$ such that
$$
0=u(x_0)=\varphi(x_0), \qquad \varphi(x) < u(x) \text{ if } x\neq x_0
$$
it holds that
$$
0\geq -(-\Delta_g)^s \varphi(x_0)=
2\text{p.v.} \int_{\R^n} g\left(\frac{  |\varphi(y)| }{|x_0-y|^s}\right) \frac{\varphi(y)}{|\varphi(y)|} \frac{dy}{|x_0-y|^{n+s}}.
$$
If $\varphi\geq 0$ then it follows that $\varphi\equiv 0$, from where   $u\equiv 0$ in $\R^n$. 

If $u\not \equiv 0$, by using the continuity of $u$, we can select a test function $\varphi$ such that $0\leq \varphi \leq u$ which is positive at some point. Consequently $u\equiv 0$ or $u>0$ in $\Omega$.
\end{proof}

\section{The eigenvalue problem} \label{sec.eig}
\subsection{The first eigenvalue}
In this section we prove the existence of the eigenvalue $\lam_{1,\mu}$ for each $\mu>0$ according to definition \eqref{lam.1}, as well as some properties on it and its eigenfunction.

\begin{prop} \label{propo.el}
Let $G$ be a Young function satisfying \eqref{cond}. Then, the  functional $\mathcal{F},\mathcal{G}:W^{s,G}_0(\Omega)\to\R$  defined in
\eqref{funcionales} are class $C^1$ and their  Fr\'echet derivatives $\mathcal{F}',\mathcal{G}':W^{s,G}_0(\Omega)\to (W^{s,G}_0(\Omega))'$ satisfy
$$
\langle \mathcal{F}'(u),v \rangle= \langle (-\Delta_g)^s u,v \rangle, \qquad \langle \mathcal{G}'(u),v \rangle= \int_\Omega g(|u|)\frac{u}{|u|}v\,dx
$$
for $u,v\in W^{s,G}_0(\Omega)$.
\end{prop}

\begin{proof}
For $u,v\in W^{s,G}_0(\Omega)$ and $t>0$ we compute
$$
\frac{\mathcal{F}(u+tv)-\mathcal{F}(v)}{t}= \iint_{\R^n\times \R^n} \left( \frac1t  \int_{ |D_s u|}^{ |D_su+tD_s v| } g(s)\,ds \right) \,d\mu.
$$
As $t\to 0$, $D_s u+t D_s v \to D_s u$  almost everywhere.
Now, since $g$ is increasing, for $t$ small we get
$$
\left|\frac1t \int_{ |D_s u| }^{ |D_s u+tD_s v| } g(s)\,ds \right| \leq g(|D_s u|+ |D_s v|) |D_s v|.
$$
We claim that
$g( |D_s w| )\in L^{G^*}(\R^{2n}, \,d\mu)$ for all $w\in W^{s,G}_0(\Omega)$. Indeed, by using \eqref{xxxx}, \eqref{cond} and the fact that $g^{-1}$ is increasing we obtain that
\begin{align*}
\iint_{\R^n\times \R^n} G^*(|g(| D_s w| )|) \,d\mu
&= 
\iint_{\R^n\times \R^n} \left(\int_0^{g(| D_s w| )} g^{-1}(s)\,ds \right) \,d\mu \\
&\leq 
\iint_{\R^n\times \R^n}  g^{-1}(g(|D_s w| ))  g( |D_s w| )\, \,d\mu \\
&\leq 
p^+ \iint_{\R^n\times \R^n}  G(|D_s w| ) \,d\mu\\
&=
p^+ \Phi_{s,G}(w).
\end{align*}
Then, $g(|D_s u| + |D_s v|) \in L^{G^*}(\R^{2n}, d\mu)$, and using \eqref{Young}, we get that
$$
\iint_{\R^n \times \R^n} g(|D_s u|+ |D_s v|) |D_s v| \,d\mu<\infty.
$$
Thus, by the dominated convergence theorem,
\begin{align*}
\langle \mathcal{F}'(u),v\rangle =\lim_{t\to 0} \frac{\mathcal{F}(u+tv)-\mathcal{F}(t)}{t} &= \frac{d}{dt} 	\mathcal{F}( u+tv )\Big|_{t=0}\\
&= \iint_{\R^n\times \R^n} 
g(|D_s u|)\frac{D_s u}{|D_s u|} D_s v   \,d\mu\\
&= \langle (-\Delta_g)^s u,v \rangle.
\end{align*} 
 
Now, let us see that $\mathcal{F}'$ is continuous. Let $\{u_j\}_{j\in\N}\subset W^{s,G}_0(\Omega)$ be a   such that $u_j\to u$ and observe that
$$
|\langle \mathcal{F}'(u_j)-\mathcal{F}'(u),v\rangle |= \left|\iint_{\R^n\times\R^n} \left( g(|D_s u|)\frac{D_s u}{|D_s u|}- g(|D_s u_j|)\frac{D_s u_j}{|D_s u_j|}  \right) D_s v \,d\mu\right|,
$$
then, by  Egoroff's Theorem, there exists a positive sequence  $\delta_j\to 0$ such that
\begin{align*}
\sup_{\|v\|_{s,G}\leq 1} \iint_{\R^n\times\R^n}   &\left( g(|D_s u|)\frac{D_s u}{|D_s u|}- g(|D_s u_j|)\frac{D_s u_k}{|D_s u_j|}  \right) D_s v \,d\mu \\
& \leq \left\|     g(|D_s u|)\frac{D_s u}{|D_s u|}- g(|D_s u_j|)\frac{D_s u_j}{|D_s u_k|}     \right\|_{L^{G^*}(\R^{2n},d\mu)} +\delta_k,
\end{align*}
where we have used  the H\"older's inequality for Orlicz spaces (see \cite[Theorem 3.3.8]{KJF}).
Now, since $G^*$ satisfies \eqref{g.s.1}, by Proposition \ref{teo.rr} we get
$$
\left\| g(D_s u)-g(D_s u_k) \right\|_{L^{G^*}(\R^{2n},d\mu )} \to 0,
$$
and therefore $
\|\mathcal{F}'(u_n)-\mathcal{F}'(u)\|_{(W^{s,G}_0(\Omega))'}\to 0$ as required.
 
A similar reasoning allow us to claim that $\mathcal{G}\in C^1$ and
\begin{align*}
\lim_{t\to 0} \frac{\mathcal{G}(u+tv)-\mathcal{G}(v)}{t} &= \frac{d}{dt} 	\Phi_{G}( u+tv )\Big|_{t=0}= \int_{\Omega} 
g(|u|)\frac{u}{|u|} v 
\end{align*} 
and the proof concludes.
\end{proof}

As a consequence,   we get the eigenvalue existence.

\begin{thm} \label{teo666}
Let $G$ be a Young function satisfying \eqref{cond}. Then, for every $\mu>0$ there exists a   positive eigenvalue $\lam_{1,\mu}$ of \eqref{eq.autov} with non-negative eigenfunction $u_{1,\mu}\in W^{s,G}_0(\Omega)$ such that $\mathcal{G}(u_{1,\mu})=\mu$. Moreover, $\lambda_{1,\mu}$ is bounded by below independently of $\mu$.

\end{thm}
\begin{proof}
Given a fixed value of $\mu>0$, in light of Proposition \ref{propo.1} there exists a function $u_{1,\mu}\in W^{s,G}_0(\Omega)$ attaining the minimum in \eqref{min.prob}. In view of Proposition \ref{propo.el}, from  the Lagrange multiplier rule there exists $\lam_{1,\mu}$  such that the constraint $\mathcal{G}(u_{1,\mu})=\mu$ is satisfied and
$$
\langle (-\Delta_g)^s u_{1,\mu} ,v \rangle =\lam_{1,\mu}  \int_\Omega g(|u_{1,\mu}|)\frac{u_{1,\mu}}{|u_{1,\mu}|}v \qquad \forall  v\in W^{s,G}_0(\Omega).
$$
Choosing $v=u_{1,\mu}$ in the last expression, we obtain that $\lam_{1,\mu}>0$. 

By definition, $u_{1,\mu}$ realizes the infimum in the expression of $\alpha_{1,\mu}$ defined in \eqref{min.prob}. Since the functionals $\mathcal{F}$ and $\mathcal{G}$ are invariant by replacing $u_{1,\mu}$ with $|u_{1,\mu}|$ we may assume that $u_{1,\mu}$ is one-signed in $\Omega$.

Finally, from  \eqref{cond}  and Proposition \ref{thm.poincare} we get
$$
\lam_{1,\mu} = \frac{\langle (-\Delta_g)^s u_{1,\mu} ,u_{1,\mu} \rangle}{\int_\Omega g(u_{1,\mu})u_{1,\mu} \,dx} \geq \frac{p^-}{p^+}\lam_{1,\mu} \geq  C \min\{\d^{-sp^+},\d^{-sp^-}\} >0
$$
where $C$ depends only on $s$, $n$ and $p^\pm$.
\end{proof}

\begin{cor}
Given a Young function satisfying \eqref{cond}, the quantity $\lam_1$ defined in \eqref{alpha.lam.1} is strictly positive. More precisely, 
$$
\lam_{1} \geq  C \min\{\d^{-sp^+},\d^{-sp^-}\} >0
$$
where $C$ depends only on $s$, $n$ and $p^\pm$.
\end{cor}
Theorem \ref{teo666} asserts that an eigenfunction of $\lam_{1,\mu}$  is non-negative in $\Omega$. The following result claims that in fact, it is positive in $\Omega$ whenever it is a continuous function.
\begin{thm} \label{coro.1}
Let $\Omega$ be open and bounded and let $G$ be a Young function satisfying \eqref{cond} and \eqref{cond1}. Then an eigenfunction of $\lam_{1,\mu}$, $\mu>0$,  has constant sign in $\Omega$.
\end{thm}
\begin{proof}
Fixed $\mu>0$, let $(\lam_{1,\mu},v_{1,\mu})$ be an eigenpair of \eqref{eq.autov}, i.e., $v_{1,\mu}\in W^{s,G}_0(\Omega)$ is such that
$$
\langle (-\Delta_g)^s u_{1,\mu} , v\rangle = \lam_{1,\mu}\int_\Omega g(u_{1,\mu})v \quad \text{for all } v\in W^{s,G}_0(\Omega).
$$
Since Theorem \ref{teo666} gives that $\lam_{1,\mu}>0$ and $u_{1,\mu}$ is non-negative, we get  
$$
\langle (-\Delta_g)^s u_{1,\mu} , v\rangle \geq 0 \quad \text{for all non-negative } v\in W^{s,G}_0(\Omega), 
$$
i.e., $u_{1,\mu}$ is a weak super-solution of ${(-\Delta_g)^s u =0}$. Therefore, since $u_{1,\mu}$ is not non-trivial the result follows in light of Proposition \ref{u.positiva}.
\end{proof}

Finally, we prove that $\Sigma$ is closed, from where we deduce that  $\lam_1$ is an eigenvalue of \eqref{eq.autov} as well.
\begin{prop} \label{espectro.cerrado}
The spectrum of \eqref{eq.autov} is closed.
\end{prop}
\begin{proof}
Let  $\lam_j\in \Sigma$ be such that $\lam_j\to \lam$ and let $u_j\in W^{s,G}_0(\Omega)$ be an eigenfunction associated to $\lam_j$, i.e., 
\begin{equation} \label{eqj}
\iint_{\R^n\times\R^n} g(|D_s u_j|) \frac{D_s u_j}{|D_s u_j|}  D_s v\,d\mu = \lam_j \int_\Omega g(|u_j|)\frac{u_j}{|u_j|} v \qquad \text{ for all } v\in W^{s,G}_0(\Omega).
\end{equation}
 By Proposition  \ref{teo.comp}, up to a subsequence,  there exists $u\in W^{s,G}_0(\Omega)$ such that
\begin{equation} \label{conve.r}
\begin{array}{ll}
u_j\rightharpoonup u &\text{ weakly  in }	W^{s,G}_0(\Omega),\\
u_j\to u &\text{ strongly  in }L^{G}( \Omega)\\
u_j\to u &\text{ a.e. in } \R^n
\end{array}
\end{equation}
as a consequence, 
$$
G(D_s u_j)\to G(D_s u)\quad  \text{ a.e. in } \Omega.
$$
Observe that from Lemma \ref{lemita} and \eqref{cond} we have that
\begin{align*}
\iint_{\R^n\times\R^n} G^*(|g(D_s u_j)|)  \,d\mu
\leq p^+ \Phi_{s,G}(u_j) 
\leq \frac{p^+}{p^-} \lam_j \int_\Omega g(|u_j|)|u_j| 
\leq \frac{(p^+)^2}{p^-} \lam \Phi_G(u).
\end{align*}
for $j$ big enough.
So, we can assume that $g(|D_s u_j|)\frac{D_s u_j}{|D_s u_j|} \cd \eta$ weakly in $L^{G^*}(\R^{2n},d\mu)$. Again, from \eqref{conve.r} we get
$$
g(|D_s u_j|)\frac{D_s u_j}{|D_s u_j|}\to g(|D_s u|)\frac{D_s u}{|D_s u|}\quad  \text{ a.e. in } \Omega
$$
and hence, taking limit as $j\to\infty$ in \eqref{eqj}  we can assume that $\eta = g(|D_s u|)\frac{D_s u}{|D_s u |}$ a.e., consequently
$$
\iint_{\R^n\times\R^n} g(|D_s u|)\frac{D_s u}{|D_s u|}   D_s v \, d\mu = \lam \int_\Omega g(|u|)\frac{u}{|u|} v \qquad \text{ for all } v\in W^{s,G}_0(\Omega)
$$
from where the proof concludes.
\end{proof}
\begin{rem}
In contrast with the local case, the point-wise convergence of the $s-$H\"older quotients $D_s u_j$ simplifies considerably the proof, not being necessarily the deal with the monotonicity of the operator. See \cite{FBPS} for details.
\end{rem}
\begin{cor} \label{es.autov}
The number $\lam_1$ defined in \eqref{alpha.lam.1} is an eigenvalue of \eqref{eq.autov}.
\end{cor}

\section{The minimization problem} \label{min.sec}
In this section we study  the minimization problem \eqref{min.prob} related to the Euler-Lagrange equation \eqref{eq.autov}.
 
\begin{prop} \label{propo.1}
Let $G$ be a Young function satisfying \eqref{cond}. Then, the minimization problem \eqref{min.prob} has a solution $\alpha_{1,\mu}$ for each $\mu>0$.
\end{prop}

\begin{proof}
Let $\{u_{j}\}_{j\in\N}\subset M_\mu$ be a minimizing sequence for $\alpha_{1,\mu}$, i.e., $\Phi_G(u_j) =\mu$ and 
$$
 \Phi_{s,G}(u_j)  \to \mu \alpha_{1,\mu} \quad \text{ as } j\to\infty.
$$
Let us see that $\|u_{j}\|_{s,G}$ is bounded independently of $j$. If $\|u_{j}\|_{s,G}\leq 1$ there is nothing to prove. Assume that $\|u_{j}\|_{s,G}\geq 1+\ve$ for some $\ve>0$, then by using \eqref{G1} we obtain that
\begin{align*}
\Phi_{s,G}(u_{j}) &\geq   \Phi_{s,G}\left( \frac{(1+\ve)u_{j}}{\|u_{j}\|_{s,G}} \right)  \left(\frac{\|u_{j}\|_{s,G}}{1+\ve} \right)^{p^-}\\
&\geq  
\Phi_{s,G}\left( \frac{u_{j}}{\|u_{j}\|_{s, G}} \right) \left(\frac{\|u_{j}\|_{s,G}}{1+\ve} \right)^{p^-} = 
\left(\frac{\|u_{j}\|_{s,G}}{1+\ve} \right)^{p^-}, 
\end{align*}
where  the last equality follows from the definition of the Luxemburg norm. Hence, when $\|u_{j}\|_{s,G}>1$ the sequence $\{u_j\}_{j\in\N}\in M_\mu$   is uniformly bounded for $j$ large enough:
$$
\|u_{j}\|_{s,G} \leq ( \Phi_{s,G}(u_j))^\frac{1}{p^-} <   
  (\mu \alpha_{1,\mu})^\frac{1}{p^-}.
$$
Then, by Proposition \ref{teo.comp}, up to a subsequence, there exists $u\in W^{s,G}_0(\Omega)$ such that
\begin{align*}
&u_{j}\rightharpoonup u \text{ weakly  in }	W^{s,G}_0(\Omega),\\
&u_{j}\to u \text{ strongly  in }L^{G}( \Omega) \text{ and a.e. in }\Omega,
\end{align*}
from where $\mathcal{G}(u)= \mu$ and then $u\in M_\mu$. 

Now, since the application $u \mapsto \Phi_{s, G} (u)$ is lower semi-continuous 
due to the convexity of the modular, by the Fatou's lemma we get 
$$
\Phi_{s,G} (u)  \leq \liminf_{j\to\infty} \Phi_{s,G} (u_j)  =\mu \alpha_{1,\mu} .
$$
Since by definition $\mu \alpha_{1,\mu} \leq \Phi_{s, G} (u)$, the result follows.
\end{proof}

As a consequence of the Poincar\'e's inequality, namely, Corollary \ref{coro.cota.inf}, we obtain the following.
\begin{prop} \label{alpha.acotado}
The number $\alpha_{1,\mu}$ is strictly positive. Moreover, 
$$
\alpha_{1,\mu} \geq   C \min\{\d^{-s p^-},\d^{-s p^+}\} >0
$$
where $C=C( n,s,p^\pm)$.
\end{prop}

As a direct consequence \eqref{cond} and Proposition \ref{alpha.acotado} we get the following result.
\begin{cor} \label{prop.cota.inf}
The quantities  $\lam_1$ and $\alpha_1$ defined in \eqref{alpha.lam.1} are  strictly positive and comparable. More precisely, 
$$
C \min\{\d^{-s p^-},\d^{-s p^+}\} \leq \frac{p^-}{p^+}  \alpha_1  \leq \lambda_1 \leq \frac{p^+}{p^-}  \alpha_1
$$
where $C=C( n,s,p^\pm)$.
\end{cor}

\subsection{Continuity with respect to $\rho$}

In this subsection we prove continuity of the numbers $\alpha_{1,\mu}(\rho)$ defined in \eqref{alfas} with respect to $\rho$, and in the case of periodic weights we obtain estimates on the rate of convergence.

Without any additional assumption on the weight functions we prove the following.
\begin{thm} \label{sin.orden}
Let $\{\rho_\ve\}_{\ve>0}$ be a sequence of functions satisfying \eqref{cond.rho} such that $\rho_\ve \cd \rho_0$ weakly* in $L^\infty(\Omega)$. Let $\alpha_{1,\mu}(\rho_\ve)$ and $\alpha_{1,\mu}(\rho_0)$ be the numbers defined in \eqref{alfas}. Then
$$
\lim_{\ve \to 0}\alpha_{1,\mu}(\rho_\ve) = \alpha_{1,\mu} (\rho_0).
$$
\end{thm}

The proof is based in the following convergence result.
\begin{lema} \label{lema.sin.orden}
	Let $\Omega\subset \R^n$ be an open and bounded domain and $G$ a Young function. Let $\{\rho_\ve\}_{\ve>0}$ be a sequence of functions satisfying \eqref{cond.rho} such that $\rho_\varepsilon \cd \rho_0$ weakly* in $L^\infty(\Omega)$. Then 
	$$\lim_{\ve\to 0} \int_\Omega (\rho_\ve- \rho_0) G(u) =0$$
	for every $u\in W^{s,G}(\Omega)$, $0<s<1$.
\end{lema}
\begin{proof}
	The weak* convergence of $\{\rho_\varepsilon\}_{\ve>0}$ in $L^\infty(\Omega)$ says that  $\int_\Omega \rho_\ve \varphi \to \int_\Omega \rho_0 \varphi$ for all $\varphi \in L^1(\Omega)$. In particular, since $u\in W^{s,G}(\Omega)$, we have that $G(u)\in L^1(\Omega)$ and the result is proved.
\end{proof}

\begin{proof}[Proof of Theorem \ref{sin.orden}]
Let $v_\ve\in W^{s,G}_0(\Omega)$ be a minimizer of $\alpha_{1,\mu}(\rho_\ve)$. Since $v$  is admissible in the characterization of $\alpha_{1,\mu} (\rho_0)$ we have that
$$
 \alpha_{1,\mu} (\rho_0) \leq \frac{\mathcal{F}(v_\ve)}{\int_\Omega \rho_\ve G(|v_\ve|)\,dx} \frac{\int_\Omega \rho_\ve G(|v_\ve|)\,dx}{\int_\Omega \rho_0 G(|v_\ve|)\,dx} =  \alpha_{1,\mu} (\rho_\ve)\frac{\int_\Omega \rho_\ve G(|v_\ve|)\,dx}{\int_\Omega \rho_0 G(|v_\ve|)\,dx}
$$
From Lemma \ref{lema.sin.orden} we get
$$
\frac{\int_\Omega \rho_\ve G(|v_\ve|)\,dx}{\int_\Omega \rho_0 G(|v_\ve|)\,dx}  = 1+o(1)
$$
from where we obtain that
\begin{equation} \label{ex1.1}
 \alpha_{1,\mu} (\rho_0) - \alpha_{1,\mu} (\rho_\ve) \leq  o(1) \alpha_{1,\mu} (\rho_\ve).
\end{equation}
Interchanging the roles of $\alpha_{1,\mu} (\rho_0)$ and $\alpha_{1,\mu} (\rho_\ve)$,  similarly can be obtained that
\begin{equation} \label{ex2.1}
 \alpha_{1,\mu} (\rho_\ve) - \alpha_{1,\mu} (\rho_0)\leq  o(1) \alpha_{1,\mu} (\rho_0).
\end{equation}
From \eqref{ex1.1}, \eqref{ex2.1} and \eqref{cond.rho} we obtain that
$$
|\alpha_{1,\mu} (\rho_\ve)-\alpha_{1,\mu} (\rho_0)| \leq o(1) \max\{\alpha_{1,\mu} (\rho_\ve),\alpha_{1,\mu} (\rho_0)\} \leq  o(1) \frac{\alpha_{1,\mu} }{\rho_-}
$$
and the proof concludes.
\end{proof}

When the family $\{\rho_\ve\}_{\ve>0}$  is defined in terms of a $Q-$periodic function $\rho$ satisfying \eqref{cond.rho} as $\rho_\ve(x) = \rho(\tfrac{x}{\ve})$ for any $x\in \R^n$, being $Q$ the unit cube in $\R^n$,   it is well-known that $\rho_\ve \cd \bar \rho:=\pint_Q \rho$ weakly* in $L^\infty$ as $\ve\to 0$. In this case, more information about the convergence of the minimizers defined in \eqref{alfas} can be obtained.

\begin{thm} \label{con.orden}
Let $\alpha_{1,\mu}(\bar \rho)$ and $\alpha_{1,\mu} (\rho_\ve)$ be the numbers defined in \eqref{alfas}. Then, there exists a positive constant $C=C(p^+,\rho_+,s,n,\Omega)$ such that
$$
|\alpha_{1,\mu} (\rho_\ve) -\alpha_{1,\mu} (\bar \rho) | \leq C \ve^{sp^+} (\alpha_{1,\mu} )^2.
$$
\end{thm}

The proof of Theorem \ref{con.orden} is based on the following key lemma.
\begin{lema} \label{lema.clave}
Let $\Omega\subset \R^n$ be a bounded domain, $G$ a Young function satisfying \eqref{cond} and denote by $Q$ the unit cube in $\R^n$. Let $\{\rho_\ve\}_{\ve>0}$ be a sequence  defined as $\rho_\ve(x) = \rho(\tfrac{x}{\ve})$ in terms of a $Q-$periodic function $\rho$ satisfying \eqref{cond.rho} such that $\bar \rho=0$. Then, there exists $C=C(\rho_+,p^+,n,\d)>0$ such that
	$$
	\Big|\int_{\Omega} \rho_\ve G(|v|)\Big| \le C \ve^{sp^+} \Phi_{s,G}(v)
	$$
	holds for every $v\in W^{s,G}_0(\Omega)$ with $s\in(0,1)$.
\end{lema}

\begin{proof}
	Denote by $I^\ve$ the set of all $z\in \Z^n$ such that $Q_{z,\ve}\cap \Omega \neq \emptyset$, $Q_{z,\ve}:=\ve(z+Q)$. Given $v\in W^{s,G}_0(\Omega)$ we consider the function $\bar v_\ve$ given by the formula
	$$
		\bar{v}_\ve (x)=\frac{1}{\ve^n}\int_{Q_{z,\ve}} v(y)\,dy
	$$
	for $x\in Q_{z,\ve}$. We denote by $\Omega_1 = \bigcup_{z\in I^\ve} Q_{z,\ve} \supset \Omega$. Thus, we can write
	\begin{align} \label{keq0.x}
	\begin{split}
		\left|\int_{\Omega} \rho_\ve G(|v|)\right|  &=  \left|\int_{\Omega_1} \rho_\ve (G(|v|)-G(\bar{v}_\ve)) + \int_{\Omega_1} \rho_\ve G(|\bar{v}_\ve|) \right|\\
		&\leq 
		\int_{\Omega_1} \rho_\ve |(G(|v|)-G(|\bar{v}_\ve|)| + \left|\int_{\Omega_1} \rho_\ve G(|\bar{v}_\ve|) \right| := (i)+(ii).
		\end{split}
	\end{align}

	We can split $(i)$ as follows
	\begin{align} \label{dos}
		(i) &= \int_{I_1}  \rho_\ve(G(|v|)-G(|\bar{v}_\ve|)) + \int_{I_2}  \rho_\ve(G(|\bar{v}_\ve|)-G(|v|))
	\end{align}
	where $I_1=\{x\in\Omega_1 : G(|v|) - G(|\bar v_\ve|) \geq 0\}$ and $I_2=\{x\in\Omega_1 : G(|v|) - G(|\bar v_\ve|) < 0\}$. 
	
	Observe that from \eqref{G3}, Young's inequality \eqref{Young} and Lemma \ref{lemita} we get
	\begin{align*}
		 G(|v|)-G(|\bar v_\ve|) &\leq g(|v|) |v- \bar v_\ve|\\ 
		 &\leq G^*(g(|v|))+ G(|v-\bar v_\ve|)\\
		 &\leq p^+ G(|v|)+ G(|v-\bar v_\ve|)
	\end{align*}
and similarly, 
$$
G(|\bar v_\ve|)-G(|v|) \leq p^+ G(|\bar v_\ve|)+ G(|v-\bar v_\ve|).
$$
So, in light of  \eqref{dos} and \eqref{cond.rho} we have that
$$
(i)\leq p^+ \left( \int_{I_1} \rho_\ve  G(|v|) + \int_{I_1} G(|v-\bar v_\ve|) + 
\int_{I_2} \rho_\ve  G(|\bar v_\ve|) + \int_{I_2} G(|v-\bar v_\ve|) \right).
$$
Adding and subtracting $\bar v_\ve$ in the first integral and using the $\Delta_2$, from the last inequality we get
$$
(i)\leq p^+\left( \C\int_{I_1} \rho_\ve  G(|v-\bar v_\ve|) +  \C\int_{I_1} \rho_\ve  G(|\bar v_\ve|)  +    \int_{I_1} G(|v-\bar v_\ve|) + 
\int_{I_2} \rho_\ve  G(|\bar v_\ve|) + \int_{I_2} G(|v-\bar v_\ve|)  \right)
$$
and since the integrands are positive we can enlarge the domain of integration to obtain 
\begin{equation} \label{ec.interm}
(i)\leq p^+\left( (2+\rho_+ \C)\int_{\Omega_1}   G(|v-\bar v_\ve|) +  (1+\C)\int_{\Omega_1} \rho_\ve  G(|\bar v_\ve|)    \right).
\end{equation}

Now, by using Lemma \ref{poincarelema} we have
	\begin{align} \label{xxxx1}
	\begin{split}
		\int_{\Omega_1}  G(|v-\bar{v}_\ve|)
		&=
			\sum_{z\in I^\ve} \int_{Q_{z,\ve}} G(|v-\bar{v}_\ve|) dx \\
		&\leq 
			c\ve^{sp^+} \sum_{z\in I^{\ve}}   \iint_{Q_{z,\ve}\times Q_{z,\ve}} G(|D_s v|) \,d\mu  \\
	&\leq
			c\ve^{sp^+} \sum_{z\in I^{\ve}} \sum_{\tilde z\in I^{\ve}}   \iint_{Q_{z,\ve}\times Q_{\tilde z,\ve}}  G(|D_s v|) \,d\mu  \\			
		&=c\ve^{sp^+}  \iint_{\Omega_1 \times \Omega_1}  G(|D_s v|) \,d\mu   \\			
					& \le 
		c\ve^{sp^+}  \iint_{\R^n \times \R^n}  G(|D_s v|) \,d\mu.	
	\end{split}	
	\end{align}
  
	Finally, since $\bar \rho = 0$ and since $\rho$ is $Q-$periodic, we get
	\begin{equation}\label{ultima}
		\int_{\Omega_1} \rho_\ve G(|\bar v_\ve|)
			= 
		\sum_{z\in I^\ve} G(|\bar v_\ve|) \int_{Q_{z,\ve}} \rho_\ve 
			= 
		0.
	\end{equation}
	Therefore, combining \eqref{keq0.x}, \eqref{ec.interm}, \eqref{xxxx1} and \eqref{ultima} we find that $(ii)=0$ and
	$$
		\Big|\int_{\Omega} \rho_\ve G(|v|)\Big| \le c p^+(2+\rho_+ \C) \ve^{sp^+} \Phi_{s,G}(v)
	$$
	and the proof finishes. 
\end{proof}

\begin{proof}[Proof of Theorem \ref{con.orden}]
The proof runs similarly to those of Theorem \ref{sin.orden} by using Lemma \ref{lema.clave} instead of Lemma \ref{lema.sin.orden}.
Indeed, let $v_\ve\in W^{s,G}_0(\Omega)$ be a minimizer of $\alpha_{1,\mu}(\rho_\ve)$. Since $v$  is admissible in the characterization of $\alpha_{1,\mu} (\bar \rho)$ we have that
$$
\alpha_{1,\mu} (\bar \rho)\leq \frac{\mathcal{F}(v_\ve)}{\int_\Omega \rho_\ve G(|v_\ve|)\,dx} \frac{\int_\Omega \rho_\ve G(|v_\ve|)\,dx}{\int_\Omega \bar \rho G(|v_\ve|)\,dx} = \alpha_{1,\mu} (\rho_\ve)\frac{\int_\Omega \rho_\ve G(|v_\ve|)\,dx}{\int_\Omega \bar \rho G(|v_\ve|)\,dx}.
$$
From Lemma \ref{lema.sin.orden} and \eqref{cond.rho} we get
$$
\frac{\int_\Omega \rho_\ve G(|v_\ve|)\,dx}{\int_\Omega \bar \rho G(|v_\ve|)\,dx}  = 1+C \ve^{sp^+} \frac{\Phi_{s,G}(v_\ve)}{\int_\Omega \bar \rho G(|v_\ve|)\,dx} \leq  1+C \ve^{sp^+}
\frac{\rho_-}{\bar \rho}\frac{\Phi_{s,G}(v_\ve)}{\int_\Omega \rho_\ve G(|v_\ve|)\,dx}
$$
from where we obtain that
\begin{equation} \label{ex1}
\alpha_{1,\mu} (\bar \rho) -\alpha_{1,\mu} (\rho_\ve) \leq  C \ve^{sp^+} (\alpha_{1,\mu}(\rho_\ve) )^2.
\end{equation}
Interchanging the roles of $\alpha_{1,\mu} (\bar \rho)$ and $\alpha_{1,\mu}(\rho_\ve)$,  similarly can be obtained that
\begin{equation} \label{ex2}
\alpha_{1,\mu} (\rho_\ve) -\alpha_{1,\mu} (\bar \rho) \leq C \ve^{sp^+} (\alpha_{1,\mu} (\bar \rho))^2.
\end{equation}
From \eqref{ex1}, \eqref{ex2} and \eqref{cond.rho} we obtain that
$$
|\alpha_{1,\mu} (\rho_\ve)-\alpha_{1,\mu} (\bar \rho)| \leq C \ve^{sp^+} \max\{\alpha_{1,\mu} (\rho_\ve))^2,(\alpha_{1,\mu} (\bar \rho))^2\} \leq C \ve^{sp^+} \frac{(\alpha_{1,\mu} )^2}{(\rho_-)^2},
$$
which concludes the proof.
\end{proof}

\section{Further properties} \label{sec.adic}
\subsection{Nodal domains}
In this subsection we need more regularity on the Young function $G$, namely, we assume the renowned \emph{Lieberman's condition}
\begin{equation} \label{condLib} \tag{L'}
p^- -1 \leq \frac{tg'(t)}{g(t)} \leq p^+-1  \quad \forall t>0
\end{equation}
for certain constant $1<p^-<p^+<\infty$.

Observe that condition \eqref{condLib} on $G$ implies \eqref{cond}.

The following auxiliary lemma is useful for our next result.
\begin{lema} \label{lemmita}
Let $G$ be a Young function satisfying \eqref{condLib}  such that $g=G'$. Then, the function $h(t)=\frac{1}{t}g(\frac{t}{c})$ is increasing for any fixed $c>0$.
\end{lema}
\begin{proof}
Since $h'(t)=g'\left(\tfrac{t}{c}\right)\frac{1}{ct}- g\left(\frac{t}{c} \right)\frac{1}{t^2}$, $h$ is increasing if $\frac{g'(t/c)}{g(t/c)}\geq \frac{c}{t}$ which is guaranteed by \eqref{condLib}.
\end{proof}

\begin{prop}\label{prop.cota.inf.1}
Let $\lam(\Omega)$ be an eigenvalue of \eqref{eq.autov} with $G$ satisfying \eqref{condLib} with continuous sign-changing eigenfunction $u$. Then
$$
p^+\lam(\Omega)>\lam_1(\Omega^+), \quad p^+\lam(\Omega)>\lam_1(\Omega^-)
$$
holds for the open sets $\Omega^+=\{u>0\}$ and $\Omega^-=\{u<0\}$. 

Moreover, if $\Omega$ has its diameter comparable with its Lebesgue measure, then $$
\lam(\Omega)\geq C(n,s,p^\pm) |\Omega^\pm|^{-\gamma}
$$
for some constant $\gamma>0$ depending on $p^\pm$, $s$ and $n$.
\end{prop}
 
\begin{proof}
We decompose $u=u^+ - u^-$ where $u^\pm=\max\{\pm u,0\}$ denote the positive and negative part of $u$, respectively. Observe that
$$
(u(x)-u(y))(u^+(x)-u^+(y))=|u^+(x)-u^+(y)|^2 + u^+(x)u^-(y)+ u^+(y)u^-(x).
$$
Hence, choosing $u^+\in W^{s,G}_0(\Omega)$ as a test in \eqref{eq.autov.debil}, from the identity above  we find that  
\begin{align*}
\lam \int_\Omega g(|u|)u^+\,dx &=  \iint_{\R^n\times\R^n} g\left(|D_s u|\right) \frac{u(x)-u(y)}{|u(x)-u(y)|} \frac{u^+(x)-u^+(y)}{|x-y|^s}\,d\mu\\
&=   \iint_{\R^n\times\R^n} g\left(|D_s u|\right) \frac{|u^+(x)-u^+(y)|^2}{|u(x)-u(y)|} \frac{dx\,dy }{|x-y|^{n+s}}\\
&+2\iint_{\R^n\times\R^n} g\left(|D_s u|\right) \frac{u^+(x) u^-(y)}{|u(x)-u(y)|} \frac{dx\,dy }{|x-y|^{n+s}}:= I + II.
\end{align*}	
Now, since 
$$
|u(x)-u(y)|^2 = |u^+(x)-u^+(y)|^2 + |u^-(x)-u^-(y)|^2 + 2u^+(x) u^-(y)+2u^+(y) u^-(x),
$$
we get that $|u(x)-u(y)|\geq |u^+(x)-u^+(y)|$. Therefore, from Lemma \ref{lemmita} we get
\begin{align*}
I&\geq  \iint_{\R^n\times\R^n} g\left(|D_s u^+|\right) |D_s u^+| \,d\mu := I'.
\end{align*}
Moreover, the identity above also gives that $|u(x)-u(y)|\geq \sqrt{2u^+(x) u^-(y)}$, from where
$$
II\geq \iint_{\R^n\times\R^n} g\left(\frac{\sqrt{2 u^+(x) u^- (y)}}{|x-y|^s}\right) \frac{u^+(x) u^-(y)}{|u(x)-u(y)|} \frac{dx\,dy }{|x-y|^{n+s}}:=II' >0,
$$
since $g$ is increasing. Consequently,  we find that 
$$
\lam(\Omega) \int_{\Omega^+} g(u^+)u^+\,dx \geq I'+ II' > I'.
$$
and since $u^+\in W^{s,G}_0(\Omega^+)$   the last expression yields    $\lam(\Omega)>\lam_1(\Omega^+)$. Finally, if $\d$ and $|\Omega|$ are comparable, from Corollary \ref{prop.cota.inf} we conclude that
$$
\lam(\Omega)>\lam_1(\Omega^+)\geq  C(n,s,p^\pm) |\Omega^+|^{-\gamma}
$$
for some suitable positive number $\gamma$.
The proof for $\Omega^-$ is analogous.
\end{proof}

\subsection{Behaviour of $\alpha_{1,\mu}$ as $s\to 1$}
As a direct implication of the $\Gamma-$convergence of  modulars stated in \cite{FBS}, the behavior of the Poincar\'e constant \eqref{min.prob}
as $s\to 1^+$ can be characterized. For definitions and an introduction to the $\Gamma-$convergence theory, see for instance \cite{Dalm}. 

In this paragraph it will be convenient to empathize the dependence on $s$ in $\alpha_{1,\mu}$ and in the set $M_\mu$. Given an open and bounded set $\Omega\subset\R^n$, a parameter $s\in(0,1)$ and a Young function $G$ satisfying \eqref{cond}, we consider the fractional minimizer and the limit minimizer defined as
\begin{equation} \label{min.prob.s1}
\alpha_{1,\mu,s}=\inf_{u\in M_{\mu,s}}  \frac{\Phi_{s,G}(u)}{\Phi_G(u)} ,
\qquad 
\alpha_{1,\mu,1}=\inf_{u\in \tilde M_\mu}  \frac{\Phi_{\tilde G}(|\nabla u|)}{\Phi_{\tilde G}(u)} 
\end{equation}
where for $\mu>0$ we consider the sets
$$
M_{\mu,s}=\{u\in W^{s,G}_0(\Omega)\colon \Phi_G(u)=\mu\}, \qquad \tilde M_{\mu,s}=\{u\in W^{1,\tilde G}_0(\Omega)\colon \Phi_{\tilde G}(u)=\mu\}
$$
and the limit Young function $\tilde G$ is defined as follows
$$
\tilde G(t):=\lim_{s\uparrow 1} (1-s)\int_0^1  \int_{\mathbb{S}^{n-1}}   G\left( t |z_n| r^{1-s}\right)   dS_z \frac{dr}{r},
$$
see \cite[Proposition 2.16]{FBS} for details.

In this context we define the energy functionals $\J_s,\J\colon L^G(\Omega)\to \R\cup \{+\infty\}$ by
\begin{align*}
	\J_s(u)=
	\begin{cases}
	(1-s)\frac{1}{\mu}\Phi_{s,G}(u) &\text{if } u\in M_{\mu,s}\\
	+\infty &\text{otherwise}
	\end{cases}, \quad \J(u)=
	\begin{cases}
	\frac{1}{\mu}\Phi_{\tilde G}(|\nabla u|) &\text{if } u\in \tilde M_\mu \\
	+\infty &\text{otherwise}.
	\end{cases}
\end{align*}

 \begin{prop} \label{sto1}
$$
\lim_{s\to 1^+}(1-s)\alpha_{1,\mu,s} = \alpha_{1,\mu,1},
$$
\end{prop}
\begin{proof}

By \cite[Theorem 6.5]{FBS} the functional $\J_s$ $\Gamma-$converges to $\J$ as $s\to 1^+$.  The main feature of the $\Gamma-$convergence is that it implies the convergence of minima (see \cite[Theorem 7.4]{Dalm}):
$$
\lim_{s\to 1^+}\min_{L^G(\Omega)} \J_{s}(u) = \min_{L^G(\Omega)} \J(u) =  \min_{L^{\tilde G}(\Omega)} \J(u),
$$
where the last equality follows since   $G$ and $\tilde G$ define the same Orlicz space in light of \cite[Proposition 2.16]{FBS}. 
\end{proof}

\section*{Acknowledgements}
This paper is partially supported by grants UBACyT 20020130100283BA, CONICET PIP 11220150100032CO and ANPCyT PICT 2012-0153. The author is member of CONICET.
 

\begin{thebibliography}{9}

\bibitem{AlHu}
Allegretto, W. and Huang, Y. X. . \textit{Eigenvalues of the indefinite -weight $p-$Laplacian in weighted spaces}. Funkcial. Ekvac, 38(2), 233-242. (1995)
\label{AlHu}

\bibitem{An}
Anane, A., \textit{Simplicit\'e et isolation de la premi\'ere valeur propre du $p-$laplacien avec poids}, C.R. Acad. Sci. Paris Sér. I Math. 305, 725–728, (1987).
\label{An}

\bibitem{AnTs}
Anane, A., Tsouli, N. \textit{On the second eigenvalue of the p-Laplacian}, Nonlinear partial differential equations (Fés, 1994), 1-9. Pitman Res. Notes Math. Ser, 343.
\label{AnTs}

\bibitem{CIL}
Crandall, M., Ishii, H. and Lions, P.L., \textit{User’s guide to viscosity solutions of second order partial differential equations}. Bull. Amer. Math. Soc. 27, 1–67, (1992).
\label{CIL}

\bibitem{Dalm}
Dal Maso, G., \textit{An introduction to $\Gamma-$convergence}, Progress in Nonlinear Diff. Eq. and their Applications, vol. 8, Birkhauser Boston, Inc., Boston, MA, (1993).
\label{Dalm}

\bibitem{hit}
Di Nezza, E., Palatucci, G. and Valdinoci, E., \textit{Hitchhiker's guide to the fractional Sobolev spaces}. Bulletin des Sciences Math\'ematiques, 136(5), 521-573, (2012).
\label{hit}

\bibitem{FBS}
Fern\'andez Bonder, J. and Salort, A., \textit{Fractional order Orlicz-Sobolev spaces}, preprint  arXiv:1707.03267 
\label{FBS} 

\bibitem{FBPLS}
Fern\'andez Bonder, J., P\'erez Llanos, M. and Salort, A., \textit{A H\"older Infinity Laplacian obtained as limit of Orlicz Fractional Laplacians},  preprint  arXiv:1807.01669
\label{FBPLS}

\bibitem{FBPS}
Fern\'andez Bonder, J. and Salort, A., \textit{Quasilinear eigenvalues}.  Rev. Un. Mat. Argentina, 56(1) 1-25. (2015)
\label{FBS} 

\bibitem{GLMS}
Garc\'ia-Huidobro, M., Le, V. K., Man\'asevich, R., and Schmitt, K., \textit{On principal eigenvalues for quasilinear elliptic differential operators: an Orlicz-Sobolev space setting}. NoDEA, 6(2), 207-225, (2009).
\label{GLMS}
    
\bibitem{GM}
Gossez, J. P. and  Mansevich, R., \textit{On a nonlinear eigenvalue problem in Orlicz-Sobolev spaces}. Proc. of the Royal Society of Edinburgh Section A: Mathematics, 132(4), 891-909, (2002).
\label{GM}
 
\bibitem{IL}
Ishii, H. and Lions, P.L., \textit{Viscosity solutions of fully nonlinear second-order elliptic partial differential equations}. Journal of Differential equations, 83(1), 26-78, (1990).
\label{IL}

\bibitem{KLP}
Kawohl, B., Lucia, M. and Prashanth, S., \textit{Simplicity of the principal eigenvalue for indefinite quasilinear problems}. Advances in Differential Equations, 12(4), 407-434. (2007)
\label{KLP}
 
\bibitem{KR}
Krasnoselskii, M. and Rutitskii, I., \textit{Convex functions and Orlicz spaces}.  (1961).
\label{KR}

\bibitem{KJF} Kufner, A., John, O. and Fucik, S., \textit{Function spaces}(Vol. 3). Springer Science Business Media. (1979).
\label{KJF}

\bibitem{K}
Kwa\'anicki, M. \textit{Eigenvalues of the fractional Laplace operator in the interval}. Journal of Functional Analysis, 262(5), 2379-2402. (2012)
\label{K}
  
\bibitem{LL}
Lindgren, E. and Lindqvist, P., \textit{Fractional eigenvalues}. Calculus of Variations and Partial Differential Equations, 49(1-2), 795-826, (2014).
\label{LL}

\bibitem{Lindq}
Lindqvist, P. \textit{On the equation $div(|\nabla u|^{p-2}\nabla u) +\lambda |u|^{p-2}u=0$}. Proceedings of the American Mathematical Society, 157-164. (1990)
\label{Lindq}

\bibitem{Montenegro}
Montenegro, M. and Lorca, S. \textit{The eigenvalue problem for quasilinear elliptic operators with general growth}. Applied Mathematics Letters, 25(7), 1045-1049. (2012)
\label{Montenegro}

\bibitem{MT}
Mustonen, V. and Tienari, M., \textit{An eigenvalue problem for generalized Laplacian in Orlicz—Sobolev spaces}. Proc. of the Royal Society of Edinburgh Sect. A, 129, 153-163, (1999).
\label{MT}
 
\bibitem{RR}
Rao, M. and Ren, Z, \textit{Applications of Orlicz spaces} (Vol. 250). CRC Press. (2002).
\label{RR}

\bibitem{SV}
Servadei, R.,  Valdinoci, E. (2014). On the spectrum of two different fractional operators. Proceedings of the Royal Society of Edinburgh Section A: Mathematics, 144(4), 831-855.
\label{SV}
 
\bibitem{T}
Tienari, M., \textit{Ljusternik-Schnirelmann theorem for the generalized Laplacian}. Journal of Differential Equations, 161(1), 174-190, (2000).
\label{T}
 

\end{thebibliography}

\end{document}